\documentclass[a4,12pt]{amsart}
\usepackage{amssymb,amstext,amsmath,amscd,amsthm,amsfonts,enumerate,latexsym, comment}
\usepackage{color}
\usepackage[dvipdfmx]{graphicx}

\usepackage[all]{xy}
\usepackage{hyperref} 
\theoremstyle{plain}
\newtheorem{thm}{Theorem}[section]

\newtheorem*{thm*}{Theorem}
\newtheorem*{cor*}{Corollary}

\newtheorem{prop}[thm]{Proposition}

\newtheorem{lem}[thm]{Lemma}
\newtheorem{cor}[thm]{Corollary}

\newtheorem*{claim*}{Claim}

\theoremstyle{definition}

\newtheorem{ex}[thm]{Example}

\newtheorem{rem}[thm]{Remark}

\theoremstyle{remark}

\numberwithin{equation}{thm}

\def\Min{\operatorname{Min}}

\def\Spec{\operatorname{Spec}}

\def\m{\mathfrak m}

\newcommand{\rmQ}{\mathrm{Q}}

\newcommand{\jcb}{\operatorname{J}}

\newcommand{\mapright}[1]{%
\smash{\mathop{%
\hbox to 1cm{\rightarrowfill}}\limits^{#1}}}

\newcommand{\mapleft}[1]{%
\smash{\mathop{%
\hbox to 1cm{\leftarrowfill}}\limits_{#1}}}

\def\depth{\operatorname{depth}}

\def\Ass{\operatorname{Ass}}

\def\Spec{\operatorname{Spec}}
\title[When is the strict closure of rings finitely generated?]{When is the strict closure of rings finitely generated?}

\author[R. Isobe]{Ryotaro Isobe}

\address{Faculty of Education, Chiba University, 1-33, Yayoi-cho, Inage-ku, Chiba-shi, Chiba-ken, 263-8522, Japan}
\email{ryotaro.isobe@faculty.gs.chiba-u.jp}


\thanks{The author was supported by JSPS KAKENHI Grant Number JP24K16910.}

\begin{document}

\maketitle

\setlength{\baselineskip} {15.2pt}

\begin{abstract}
This paper investigates the finite generation of the strict closure of rings in arbitrary dimension. For a Noetherian local ring $(R, \m)$, we provide a sufficient condition under which the strict closure $R^*$ is finitely generated as an $R$-module. Using this result, we characterize the finite generation of the strict closure over excellent rings.  

\end{abstract}

%{\footnotesize \tableofcontents}

%%%%%%%%%%%%%%%%%%%%%%%%%%%%%%%%%%%%%%%%%%%%%%%%%%%%%%%%%%%%%%%%%%%%%%%%%%%%%%%%%%%%%%%%%%%%%%%%%%%%%%%%%%%%%%%%%%%%%%%%%%%%%%%%%%%%%%%%%%%%%%%%

\section{Introduction}\label{intro}
Let $S/R$ be an extension of commutative rings, and we set
$$
R_S^{*}=\left\{ \alpha \in S \mid \alpha \otimes 1 = 1 \otimes \alpha \text{ in } S \otimes_RS\right\}.
$$
Then, $R^*_S$ is an intermediate ring between $R$ and $S$, which is called the {\it strict closure of $R$ in $S$}. In particular, we denote it simply by $R^*$ if $S=\overline{R}$, where $\overline{R}$ denotes the integral closure of $R$ in the total ring $\rmQ(R)$ of fractions. 

The notion of the strict closure of rings was introduced by Lipman \cite{L} in 1971.
 In the same paper, it was conjectured that $R=R^*$ if and only if $R$ is an {\it Arf ring} when $R$ is a Cohen-Macaulay semi-local ring whose localization at maximal ideals are dimension one. This conjecture was partially proved in \cite{L} for the case when $R$ contains a field ({\cite[Proposition 4.5, Theorem 4.6]{L}}), and in 2023, it was settled in \cite{C} for the general case ({\cite[Theorem 4.4]{C}}). 
 Furthermore, in recent years, there has been active research on one-dimensional strictly closed rings (i.e., Arf rings), and the structures of their integrally closed ideals and reflexive modules have been clarified (see, for example, \cite{D, I1, IK}).
 In addition, \cite{I2} investigates the structure of $R^*$ when $\dim R=1$, characterizing the finite generation of $R^*$.

%As $[R^*]^*=R^*$ and $R^*\subseteq T^*$ if $T$ is an intermediate ring between $R$ and $\overline{R}$, $[-]^*$ is actually a closure operation. By contrast, since $R^*_S$ is the kernel  of the $R$-linear map $\sigma$, the notion of strict closure is compatible with flat base changes. Readers may consult \cite{L, C, EG, EGI} for further reading about the properties of strict closures.   

%which is the kernel of the  $R$-linear map, 
%$$
%\sigma: S \to S \otimes_RS, \ \  \alpha \mapsto \alpha \otimes 1 - 1 \otimes \alpha.
%$$

The aim of this article is to characterize the finite generation of $R^*$ in arbitrary dimension. For a Noetherian local ring $(R, \m)$ with $\depth R>0$, it is well-known that $R$ is reduced if the integral closure $\overline{R}$ is a finitely generated $R$-module, and $\overline{R}$ is a finitely generated $R$-module if $\widehat{R}$ is reduced, where $\widehat{R}$ denotes the $\m$-adic completion of $R$ (see, for example, {\cite[Theorem 9.2.2]{SH}}). 
In particular, $\overline{R}$ is a finitely generated $R$-module if and only if $R$ is reduced, when $R$ is an {\it excellent ring}. 

%In contrast, the author proved that the equality $(\sqrt{(0)})^2=(0)$ holds in $R$ if $R^*$ is a finitely generated $R$-module ({\cite[Proposition 4.1]{I2}}), and furthermore, the above equality in $\widehat{R}$ characterizes the finite generatedness of $R^*$ when $R$ is a one-dimensional Cohen-Macaulay semi-local ring with $\dim R_M=1$ for every maximal ideal $M$ ({\cite[Theorem 4.2]{I2}}). 

To state our main result, recall that a ring $R$ {\it satisfies Serre's condition $(S_1)$} if every associated prime ideal of $R$ is minimal in $\Spec R$.
The main result of this article is the following, which is the strict closure analogue of these above results on the integral closure, and a kind of generalization of {\cite[Theorem 4.2]{I2}} to arbitrary dimension. 

\begin{thm}[$=$ Theorem \ref{mthm} and Corollary \ref{2.5}]\label{thm1}
Suppose that $(R, \m)$ is a Noetherian local ring with $\depth R>0$. Then, the following hold true.

\begin{enumerate}[$(1)$]
\item
 If $\widehat{R}$ satisfies Serre's condition $(S_1)$ and $(\sqrt{(0)})^2=(0)$ in $\widehat{R}$, then $R^*$ is a finitely generated $R$-module. 
\item
Suppose that $R$ is excellent and satisfies Serre's condition $(S_1)$. Then, $R^*$ is a finitely generated $R$-module if and only if $(\sqrt{(0)})^2=(0)$ in $R$.

\end{enumerate}

\end{thm}

%\begin{center}
%\begin{enumerate}[$(i)$]
%\item
%$R$ is reduced if the integral closure $\overline{R}$ is a finite generated $R$-module, and 
%\item
%$\overline{R}$ is a finitely generated $R$-module if $\widehat{R}$ is reduced, where $\widehat{R}$ denotes the $\m$-adic completion of $R$. 

%\end{enumerate}
%\end{center}

%%%%%%%%%%%%%%%%%%%%%%%%%%%%%%%%%%%%%%%%%%%%%%%%%%%%%%%%%%%%%%%%%%%%%%%%%%%%%%%%%%%%%%%%%%%%%%%%%%%%%%%%%%%%%%%%%%%%%%%%%%%%%%%%%%%%%%%%%%%%%%%%
\section{Proof of main results}\label{sec2}

In what follows, let $(R, \m)$ be a Noetherian local ring with $\depth R>0$. We denote by $\rmQ(R)$ the total ring of fractions of $R$. We  say that $R$ {\it satisfies Serre's condition $(S_1)$} if $\Ass R=\Min R$.  %We begin with the following lemma. 

Suppose that $R$ satisfies Serre's condition $(S_1)$, and set $B=R/\sqrt{(0)}$.
Since both $R$ and $B$ satisfy $(S_1)$, we have
\begin{center}
 $\displaystyle \rmQ(R)=\prod_{P\in \Min R} R_P$ \quad and \quad $\displaystyle \rmQ(B)=\prod_{P\in \Min R} R_P/PR_P$. 
 \end{center}
 Therefore, passing to the natural surjective map $\rmQ(R)\to \rmQ(B)$, we can consider that $\rmQ(B)=\rmQ(R)/\sqrt{(0)}\rmQ(R)$.
 Let $\overline{R}$ denote the integral closure of $R$ in $\rmQ(R)$. We then have the following.
 
 \begin{lem}\label{2.1}
$\overline{B}=\overline{R}/\sqrt{(0)}\rmQ(R)$.
\end{lem}
\begin{proof}
The inclusion $\overline{R}/\sqrt{(0)}\rmQ(R) \subseteq \overline{B}$ is clear. 

Let $\varphi \in \overline{B}$. Then, 
$$\varphi^n+\xi_1\varphi^{n-1}+\cdots + \xi_n=0$$
in $\rmQ(B)$ for some $n\ge 1$ and $\xi_i\in B$. We write $\varphi=\overline{\alpha}$ and $\xi_i=\overline{a_i}$ with $\alpha \in \rmQ(R)$ and $a_i\in R$, where $\overline{*}$ denotes the image of $*$ in $\rmQ(B)$. 
 Since $\alpha^n+a_1\alpha + \cdots + a_n\in \sqrt{(0)}\rmQ(R)$, there exists $m\ge 1$ such that 
 $$(\alpha^n+a_1\alpha + \cdots + a_n)^m=0$$ 
 in $\rmQ(R)$, which implies that $\alpha \in \overline{R}$.
\end{proof}

Let $\widehat{R}$ denote the $\m$-adic completion of $R$.
We obtain the following.

\begin{prop}\label{2.2}
Suppose that $R=\widehat{R}$ and $R$ satisfies Serre's condition $(S_1)$. Then, there exists an intermediate ring $S$ between $R$ and $\rmQ(R)$ such that $S$ is a finitely generated $R$-module and $\overline{R}=S+\sqrt{(0)}\rmQ(R)$.  

\end{prop}

\begin{proof}
Since $B$ is a complete reduced ring, we have that $\overline{B}$ is a finitely generated $B$-module (see, for example, {\cite[Theorem 9.2.2]{SH}}). Thus, there exists $\ell \ge 1$ and $x_1, \ldots , x_{\ell} \in \overline{R}$ such that $\overline{B}=B+\sum_{i=1}^{\ell} B\overline{x_i}$ by Lemma \ref{2.1}. 
Then, by setting $S=R[x_1, \cdots, x_{\ell} ]$, $S$ is a finitely generated $R$-module and 
\begin{center}
$S+\sqrt{(0)}\rmQ(R)/\sqrt{(0)}\rmQ(R)=\overline{B}=\overline{R}/\sqrt{(0)}\rmQ(R)$, 
\end{center}
which implies that $\overline{R}=S+\sqrt{(0)}\rmQ(R)$, as desired. 
\end{proof}

Here, we introduce the following proposition.

\begin{prop}[{\cite[Theorem 2.1]{EG}}, {\cite[Corollary 2.7]{EGI}}]\label{2.3}
Let $S$ be a commutative ring and $T$ be an intermediate ring  between $S$ and $\rmQ(S)$. Let $V$ be a non-empty subset of $T$ such that $T=S[V]$. If $fg\in S$ for all $f, g\in V$, then $S=S^{*}_T$.   
\end{prop}

We then have the following, which is the main result of this article.

%the author proved that the finite gengeratedness of $R^*$ is characterized by the equality $(\sqrt{(0)})^2=(0)$ in $\widehat{R}$, when $R$ is a Cohen-Macaulay semi-local ring with $\dim R_M=1$ for any $M\in \Max R$ (\cite[Theorem 4.2]{I2}). The following is a generalization of that characterization, which is the main result of this article. 

\begin{thm}\label{mthm}
Suppose that $(R, \m)$ is a Noetherian local ring with $\depth R>0$. If $\widehat{R}$ satisfies Serre's condition $(S_1)$ and $(\sqrt{(0)})^2=(0)$ in $\widehat{R}$, then $R^*$ is a finitely generated $R$-module. 
\end{thm}

\begin{proof}
Since $\overline{R}\otimes_R \widehat{R}$ is integral over $\widehat{R}$ in $\rmQ(\widehat{R})$, we have $\widehat{R}\subseteq \overline{R}\otimes_R \widehat{R} \subseteq \overline{\widehat{R}}$.
In contrast, let us consider the exact sequence 
$$0\to R^* \to \overline{R} \overset{\sigma}{\to} \overline{R}\otimes_R \overline{R},$$
where $\sigma(\alpha)=\alpha\otimes1-1\otimes \alpha$.
By taking the tensor product $-\otimes_R \widehat{R}$, we have the exact sequence 
$$0\to R^*\otimes_R \widehat{R} \to \overline{R}\otimes_R \widehat{R} \overset{\sigma}{\to} (\overline{R}\otimes_R \widehat{R})\otimes_{\widehat{R}}(\overline{R}\otimes_R \widehat{R}),$$ 
which implies that $R^*\otimes_R \widehat{R}=\left[\widehat{R}\right]^*_{\overline{R}\otimes_R \widehat{R}}$. Since $\overline{R}\otimes_R \widehat{R} \subseteq \overline{\widehat{R}}$, we obtain that
$$R^*\otimes_R \widehat{R}=\left[\widehat{R}\right]^*_{\overline{R}\otimes_R \widehat{R}}\subseteq \left[\widehat{R}\right]^*_{\overline{\widehat{R}}}=\left[\widehat{R}\right]^*.$$ 
Hence, if $\left[\widehat{R}\right]^*$ is a finitely generated $\widehat{R}$-module, then so is $R^*\otimes_R \widehat{R}$, which implies that $R^*$ is a finitely generated $R$-module, since $\widehat{R}$ is faithfully flat. Therefore, to prove Theorem \ref{mthm}, we may assume that $R=\widehat{R}$.

Suppose that $R=\widehat{R}$, $R$ satisfies $(S_1)$, and $(\sqrt{(0)})^2=(0)$ in $R$. Then, there exists an intermediate ring $S$ between $R$ and $\rmQ(R)$ such that $S$ is a finitely generated $R$-module and $\overline{R}=S+\sqrt{(0)}\rmQ(R)$ by Proposition \ref{2.2}. Set $V=\sqrt{(0)}\rmQ(R)$. Since
\begin{center} 
$\overline{R}=S+V=S[V]=\overline{S}$ and $fg\in V^2=(0) \subseteq S$
 \end{center}
 for every $f, g\in V$, we have $S^*=S$ by Proposition \ref{2.3}.
Therefore, we obtain that $R \subseteq R^*\subseteq S^*=S$, which implies that $R^*$ is a finitely generated $R$-module, as desired.  
\end{proof}

Here, we recall the following proposition.

\begin{prop}[{\cite[Proposition 4.1]{I2}}] \label{2.5}
Suppose that $R$ is a Noetherian ring and the Jacobson radical $\jcb(R)$ of $R$ contains a non-zerodivisor of $R$. If $R^*$ is a finitely generated $R$-module, then $(\sqrt{(0)})^2=(0)$ in $R$.
\end{prop}

%\begin{proof}

%\end{proof}

%This proposition suggests that the equality $(\sqrt{(0)})^2=(0)$ is closely related to the finite generatedness of $R^*$.

Combining Theorem \ref{mthm} and Proposition \ref{2.5}, we obtain the following corollary, which holds for {\it excellent rings}; see, for example, \cite{M} for the definition of excellent rings.

\begin{cor}\label{2.6}
Suppose that $R$ is an excellent local ring and satisfies Serre's condition $(S_1)$. Then, $R^*$ is a finitely generated $R$-module if and only if $(\sqrt{(0)})^2=(0)$ in $R$.
\end{cor}

\begin{proof}
Set $B=R/\sqrt{(0)}$. Since $B$ is an excellent reduced ring, $\widehat{B}=\widehat{R}/\sqrt{(0)}\widehat{R}$ is also reduced.
Hence, we have $\sqrt{(0)}\widehat{R}=\sqrt{\sqrt{(0)}\widehat{R}}=\sqrt{(0)\widehat{R}}$. Therefore, thanks to Theorem \ref{mthm} and Proposition \ref{2.5}, $R^*$ is a finitely generated $R$-module if and only if $(\sqrt{(0)})^2=(0)$ in $R$.
\end{proof}

\begin{cor}\label{2.6}
Suppose that $R$ is a complete Noetherian local ring and satisfies Serre's condition $(S_1)$. Then, $R^*$ is a finitely generated $R$-module if and only if $(\sqrt{(0)})^2=(0)$ in $R$.
\end{cor}

\begin{rem}

In general, $R^*$ is not necessarily a finitely generated $R$-module even if the condition $(\sqrt{(0)})^2=(0)$ is satisfied in $R$.  
When $R$ is a Cohen-Macaulay local ring of dimension one, we can prove that $[ \widehat{R} ]^*=R^*\otimes_R \widehat{R}$ (\cite[Lemma 4.3]{I2}). Therefore, $R^*$ is  finitely generated if and only if so is $[ \widehat{R} ]^*$. 
The following example shows that there exists a Cohen-Macaulay local ring $R$ such that $(\sqrt{(0)})^2=(0)$ is satisfied in $R$ but not in $\widehat{R}$.

\end{rem}

\begin{ex}[{\cite[Remark 4.5]{I2}}]
Let $k[[X, Y]]$ be a $2$-dimensional formal power series ring over a field $k$ and let $A=k[[X, Y]]/(Y{^3})$. Then $(\sqrt{(0)})^2\neq (0)$ in $A$.  However, it is known that such a ring $A$ is the completion of a Noetherian local domain $R$ (see, for example, \cite[Theorem 1]{Le}). Therefore, since $A^*$ is not a finitely generated $A$-module by Corollary \ref{2.6}, $R^*$ is not a finitely generated $R$-module but $R$ is a domain.

\end{ex}

\end{document}